\documentclass[a4paper,12pt]{article}
\usepackage{amsmath,amsthm}

\usepackage[margin=1in]{geometry}
\allowdisplaybreaks[1]

\newtheoremstyle{bthm}{\baselineskip}{\baselineskip}{\slshape}{}{\bfseries}{}{\newline }{}
\newtheoremstyle{bex}{\baselineskip}{\baselineskip}{}{}{\sffamily}{}{\newline }{}
\theoremstyle{bthm}
\newtheorem{thm}{Theorem}[section]
\newtheorem{cor}[thm]{Corollary}

\theoremstyle{bex}

\begin{document}
\begin{titlepage}
\title{Clique number and Chromatic number of a graph associated to a Commutative Ring with Unity}

\author{T. Kavaskar}
\date{{\baselineskip .2in{\footnotesize Department of Mathematics, Central University of Tamil Nadu, Thiruvarur-610005, India.\\}}
{\footnotesize t\_kavaskar@yahoo.com}}
\maketitle

\begin{abstract}
Let $R$ be a commutative ring with unity (not necessarily finite). The ring $R$ consider as a simple graph whose vertices are the elements of $R$ with two distinct vertices $x$ and $y$ are adjacent if $xy$=0 in R, where 0 is the zero element of $R$. In 1988 \cite{Beck}, Beck raised the conjecture that the chromatic number and clique number are same in a graph associated to any commutative ring with unity. In 1993 \cite{And}, Anderson and Naseer disproved the conjecture by giving a counter example of the conjecture (Note that, till date this is a only one counter example). In this paper, we find the clique number and bounds for chromatic number of a graph associated to any finite product of commutative rings with unity in terms of its factors (In perticular, if $R$ is finite which is not a local ring, we obtain the clique number and bounds for chromatic number of a graph associated to $R$ in terms of its local rings). As a consequence of our results, we construct infinitely many counter examples of the above conjecture. Also, some of the main results, proved by Beck in 1988 \cite{Beck} and Anderson and Naseer in 1993 \cite{And} are consequences of our results.   
\end{abstract}

\section{Introduction}

In this paper, we consider all rings to be commutative ring with unity (that is, multiplicative identity). Let $R$ be a ring (not necessarily finite). A non-zero element $x$ of $R$ is said to be zero-divisor if there exists an non zero element $y$ of $R$ such that $xy$=0. The ring $R$ consider as a simple graph whose vertices are the elements of $R$ with two distinct vertices $x$ and $y$ are joined by an edge if and only if $xy=$0. We use $R$ to denote both the ring as well its associated graph. The concept was introduced by Beck in 1988 \cite{Beck}, where he was mainly interested in colorings. In Beck's original definition, it is observed that the vertex 0 is adjacent to every other vertex, but every elements are adjacent only to 0. Retaining this original definition, the next decade brought little progress. However, in 1999 [2], Anderson and Livingston modified and studied the zero-divisor graph whose vertices are the nonzero zero-divisors of the commutative ring with unity. This yielded a number of basic results on zero-divisor graphs of commutative rings, posets, semigroups, lattice and semilattice.  Subsequently, research has moved in several directions, see for example in (\cite{Akb}, \cite{And1}-\cite{And3}, \cite{Aza}, \cite{Dem}-\cite{Spi}). In this paper, we consider the original definition of Beck in 1988. 

The chromatic number of a ring $R$ is the minimum number of colors needed to color the elements of $R$	so that adjacent elements of $R$ receive distinct colors and is denoted by $\chi(R)$. This means, $R$ can be partitioned into $\chi(R)$ subsets of $R$, say $V_1,\ldots,V_{\chi(R)}$, such that for $1\leq i\leq \chi(R)$, $xy\neq 0$, for all $x,y\in V_i$ (This means, each $V_i$ is an independent subset in $R$). We call $V_i$ is the color class $i$, for $1\leq i\leq \chi(R)$. Note that $|V_i|$ need not be finite.

The clique number of a ring $R$ is the maximum size of a subset $C$ of $R$ for which $xy=0$, for all $x,y\in C$ and it is denoted by $\omega(R)$. That means, $\omega(R)$ is the maximum size of a complete subgraph of $R$. Note that for any ring $R$, $\omega(R)\leq \chi(R)$. Also note that these numbers need not be finite, see \cite{Beck}. 

A non-zero element $u$ of a ring $R$ is unit in $R$ if there exists $v$ in $R$ such that $uv$=1. A ring $R$ is said to be local if it has unique maximal ideal, equivalently, the set of all non-unit elements forms an ideal in $R$, see \cite{Ati}. A subset $J$ of $R$ is nilradical of $R$ if for every $x\in J$, there exists a positive integer $n$ such that $x^n=0$ (equivalently, it is the intersection of all prime ideals of $R$, see \cite{Ati}). A ring $R$ is said to be reduced if its nilradical $J=\{0\}$. The index of nilpotency of $J$ is the least positive integer $m$ for which $J^m=\{0\}$, where $J^m=JJ\ldots J$ ($m$-times). A ring $R$ is Artin if it satisfies the descending chain condition on ideals, that is, there is no infinite descending sequence of ideals (See \cite{Ati}). We recall a Structure theorem for Artin rings in \cite {Ati}. 
\begin{thm} [Structure theorem for Artin rings see \cite{Ati}]
Every commutative Artin ring $R$ is uniquely (up to isomorphism) a finite direct product of Artin local ring. 
\end{thm}
\noindent Since every finite commutative ring $R$ is Artin, $R=R_1\times R_2\times\ldots\times R_n$, where $R_i$ is finite Artin local ring.

In \cite{Beck}, Beck defined, a ring $R$ is said to be coloring if 
$\chi(R)< \infty$ and a ring $R$ is chromatic ring if $R$ is coloring with $\chi(R)=\omega(R)$. Also he raised the following conjecture.

\noindent \textbf{Conjecture 1.2 (Beck)}. If $R$ is a coloring, then $R$ is a chromatic ring, that is, $\chi(R)=\omega(R)$.

He confirmed that the conjecture for several rings see \cite{Beck}. After four years latter, 1993, Anderson and Naseer \cite{And} disproved the above conjecture by giving a counter example as follows. They consider a local ring $R=\mathcal{Z}_4[X,Y,Z]/M$, where $M$ is the ideal generated by $\big\{X^2-2, Y^2-2, Z^2-2, Z^2, 2X, 2Y,2Z, XY, XZ, YZ-2\big\}$ and proved that 
$\omega(R)=5$ and $\chi(R)=6$. Note that, till date, this is the only one counter example of the Conjecture 1.2. In this paper, we provide infinitely many counter examples using the above local ring, see in the section 4.3. 

Also they have showed that many more rings supporting the Conjecture 1.2, see \cite{And}. Further, they showed that some of the results proved by Beck (in \cite{Beck}) are consequence of their results, see \cite{And}. In this paper, we show that these results are consequence of our results see the section 4.2.

\noindent We consider throughout this paper, all rings are coloring (not necessarily finite).

In the second section, we obtain the clique number of finite product of rings in terms of its factor rings. In the third section, we obtain bounds for chromatic number of finite product of rings in terms of its factors rings. In the forth section, some of the results of Beck's are consequence of our results. Also some main results of Anderson and Naseer \cite{And} are also consequence of our results. Furthermore, as a consequence our results, we construct infinitely many counter examples of the Beck's Conjecture 1.2.

\section{Clique number of $R$.}
We say that a subset $A$ of a ring $R$ is a clique in $R$ if $A$ is a complete subgraph of $R$ (that is, for all $x, y\in A,\ xy=0$) and a clique $A$ of $R$ is a maximum clique of $R$ if $\omega(R)=|A|$. 

For $1\leq i\leq n$, let $R_i$ be a ring and let $A_i$ be a maximum clique of $R_i$. 
Define $B_i=\{x_i\in A_i \ :\ x_i^2=0\}$ and $C_i=\{x_i\in A_i \ :\ x_i^2\neq 0\}$. Then $A_i=B_i\cup C_i$ and $B_i\cap C_i=\emptyset$. Let $\mathcal{E}_i=\{A_i \ : A_i$ is a maximum clique of $R_i\}$. Choose a maximum clique $\mathcal{A}_i=\mathcal{B}_i \cup \mathcal{C}_i$ of $R_i$ (that is, $\mathcal{A}_i\in \mathcal{E}_i$) with $|\mathcal{B}_i|=$max$_{A_i\in \mathcal{E}_i}\{|B_i| \ :\ B_i \subset A_i\}$, where $\mathcal{B}_i=\{x_i\in \mathcal{A}_i \ :\ x_i^2=0\}$ and $\mathcal{C}_i=\{x_i\in \mathcal{A}_i \ :\ x_i^2\neq 0\}$. 

\begin{thm}\label{1}
Let $\textbf{R}=R_1 \times \ R_2 \times\ldots \times R_n$, where $R_i$ is a ring, for $i,\ 1\leq i \leq n$. Then $\omega (\textbf{R})=\prod\limits_{i=1}^{n}(\omega (R_i)-|\mathcal{C}_i|)+\sum\limits_{i=1}^{n}(\omega (R_i)-|\mathcal{B}_i|)$.
\end{thm}
\begin{proof} For, $1\leq i\leq n$, let $\mathcal{A}_i=\mathcal{B}_i  \cup \mathcal{C}_i$ be a maximum clique in $R_i$ with the above properties. 
Let $\mathcal{B}=\mathcal{B}_1\times \mathcal{B}_2\times \ldots \times \mathcal{B}_n$ and  let $\mathcal{S}_i=\{0\}\times\ldots \times \{0\}\times \mathcal{C}_i \times \{0\}\times \ldots \times \{0\}$, for $1\leq i\leq n$. Then $\mathcal{B} \cup \big(\bigcup{\mathcal{S}_i}\big)$ is a clique in $\textbf{R}$ and hence $\omega(\textbf{R})\geq \prod\limits_{i=1}^n\big(\omega (R_i)-|\mathcal{C}_i|\big)+\sum\limits_{i=1}^{n}\big(\omega (R_i)-|\mathcal{B}_i|\big)$. 
 
Let $\textbf{D}$ be a maximum clique of $\textbf{R}$. Then $\textbf{0} \in \textbf{D}$, where $\textbf{0}=(0,0,\ldots,0)$ and $|\textbf{D}|=\omega(\textbf{R})$. 

\noindent \textbf{Claim 1.} For $1\leq i\leq n$, there is a maximum clique $A_i$ of $R_i$ such that $\{0\} \times \ldots \times \{0\}\times A_i\times \{0\} \times \ldots \times \{0\}\subset \textbf{D}$.

\noindent Note that if for some $i$, $1\leq i\leq n$, with $\big(R_1 \times \ldots \times R_{i-1}\times \big(R_i-\{0\}\big)\times R_{i+1} \times \ldots \times R_n\big)\cap \textbf{D}=\emptyset$, then $\textbf{D}$ would not be a maximum clique in $\textbf{R}$. Therefore, for each $i,\ 1\leq i\leq n$, there exists a non-zero element $t_i\in R_i$ such that $(x_1,\ldots,x_{i-1},t_i,x_{i+1},\ldots, x_n)\in \textbf{D}$, for some $x_j\in R_j$, for all $j\neq i$. 

\noindent \textbf{Case 1.} $R_i$ has zero-divisors.

Let $s_i\neq 0\in R_i$ be a zero-divisor of $R_i$. Then there exists a non-zero element $r_i\in R_i$ such that $s_ir_i$=0. Suppose if $t_i$ is not a zero-divisor of $R_i$, then $\big(R_1 \times \ldots \times R_{i-1}\times \big(R_i-\{0\}\big)\times R_{i+1} \times \ldots \times R_n\big)\cap \textbf{D}=\{(x_1,\ldots,t_i,\ldots, x_n)\}$ and hence $(0,\ldots,s_i,\ldots, 0)$ and $(0,\ldots,r_i,\ldots, 0)$ do not belong to $\textbf{D}$ and they are adjacent to every elements of $\textbf{D}- \{(x_1,\ldots,t_i,\ldots, x_n)\}$. Therefore, $\big(\textbf{D}-\{(x_1,\ldots,t_i,\ldots, x_n)\}\big)\cup \{(0,\ldots,s_i,\ldots, 0),$ \\
$(0,\ldots,r_i,\ldots, 0)\big\}$ forms a clique in $\textbf{R}$, which is a contradiction to $|\textbf{D}|=\omega(\textbf{R})$. Thus, $t_i$ should be a zero-divisor of $R_i$. Hence $\{0\} \times \ldots \times \{0\}\times A_i\times \{0\} \times \ldots \times \{0\}\subset \textbf{D}$, for some maximum clique $A_i$ of $R_i$ containing $t_i$, (for otherwise, $\textbf{D}$ would not be a maximum clique in $\textbf{R}$). 

\noindent \textbf{Case 2.} $R_i$ has no zero-divisors.

\noindent Then $t_i$ is not a zero-divisor in $R_i$ and $\big(R_1 \times \ldots \times R_{i-1}\times \{t_i\}\times R_{i+1} \times \ldots \times R_n\big)\cap \textbf{D}=\{(x_1,\ldots,x_{i-1},t_i,x_{i+1},\ldots, x_n)\}$. In this case we replace the element $(x_1,\ldots,x_{i-1},t_i,x_{i+1},$

\noindent  $\ldots, x_n)$ in $\textbf{D}$ by $(0,\ldots,0,t_i,0,\ldots, 0)$. Then the resulting set $\textbf{D}$ is also a maximum clique in $\textbf{R}$ such that $\{0\} \times \ldots \times \{0\}\times A_i\times \{0\} \times \ldots \times \{0\}\subset \textbf{D}$, for some maximum clique $A_i$ of $R_i$ containing $t_i$ (In this case $|A_i|$=2). 

Hence the Claim 1. 
Note that if $\textbf{b}=(b_1,\ldots,b_n)\in \textbf{D}$, then $b_i \in A_i$, for $1\leq i\leq n$ (else $A_i$ would not be a maximum clique in $R_i$ by Claim 1). As we defined earlier, for $1\leq i \leq n, B_i=\{b_i\in A_i : b_i^2=0\}$, and $C_i=\{c_i\in A_i : c_i^2 \neq 0\}$.  

\noindent Let $\textbf{B}=B_1\times B_2\times \ldots \times B_n$, and $\textbf{C}=\bigcup\limits_{i=1}^n S_i$, where $S_i=\{(0,\ldots,0,c_i,0,\ldots,0) : c_i\in C_i\}$. Then clearly, $\textbf{B}\cup \textbf{C}\subseteq \textbf{D}$. Let $\textbf{a}=(a_1,\ldots,a_n)$ be a non-zero element in $\textbf{D}$.


\noindent \textbf{Claim 2.} If $a_i^2\neq 0$, for some $i$, then $a_j= 0$, for all $j\neq i$, that is $\textbf{a}\in \textbf{C}$. 

\noindent Clearly, $a_i\neq 0$. Suppose for some $j\neq i, a_j\neq 0$. Since $\textbf{a}\in \textbf{D}$ and $(0,\ldots,0,a_i,0,\ldots,0)\neq \textbf{a}\neq \textbf{0}, \textbf{a}=(a_1,\ldots,a_i,\ldots,a_n)$ is adjacent to $(0,\ldots,0,a_i,0,\ldots,0)$ and hence $a_i^2=0$, a contradiction.


\noindent \textbf{Claim 3.} If $a_i^2= 0$ and $a_i\neq 0$, for some $i$, then $a_j^2= 0$, for all $j, 1\leq j\leq n$, that is $\textbf{a}\in \textbf{B}$.

If not, there exists $j\neq i$ with $a_j^2\neq 0$ and hence by Claim 2, we have $a_{\ell}= 0$, for all $\ell\neq j$. In particular, $a_i=0$, a contradiction. 


Therefore by Claim 2 and 3, $\textbf{D}=\textbf{B}\cup \textbf{C}$. 
Since $\textbf{B}\cap \textbf{C} = \emptyset$, $|\textbf{D}|=|\textbf{B}|+|\textbf{C}|= \prod\limits_{i=1}^n|B_i|+\sum\limits_{i=1}^n|C_i|\leq \prod\limits_{i=1}^n|\mathcal{B}_i|+\sum\limits_{i=1}^n|\mathcal{C}_i|= \prod\limits_{i=1}^{n}(\omega (R_i)-|\mathcal{C}_i|)+\sum\limits_{i=1}^{n}(\omega (R_i)-|\mathcal{B}_i|) \leq \omega (\textbf{R})$, because $\mathcal{E}_i$ is the collection of maximum clique in $R_i$ containing $A_i$ and $\mathcal{A}_i$, for $1\leq i\leq n$. 
\end{proof}

\begin{cor}
If $\textbf{R}$ is finite ring, then $R=R_1\times R_2\times\ldots\times R_n$, where $R_i$ is a finite local ring and hence $\omega (\textbf{R})=\prod\limits_{i=1}^{n}(\omega (R_i)-|\mathcal{C}_i|)+\sum\limits_{i=1}^{n}(\omega (R_i)-|\mathcal{B}_i|)$.
\end{cor}

\section{Chromatic number of $R$}

In \cite{Beck}, Beck observed that, 
$\chi(R_1\times R_2)\geq \chi(R_1)+\chi(R_2)-1$. In this section, we obtain an upper bound of $R_1\times R_2$ in terms of chromatic numbers of $R_1$ and $R_2$. 

For any two rings $R_1$ and $R_2$, let $U_1,U_2,\ldots, U_{k_1}$ be a proper coloring of $R_1$, where $k_1=\chi(R_1)$ and let $V_1,V_2,\ldots, V_{k_2}$ be a proper coloring of $R_2$, where $k_2=\chi(R_2)$. Since $0\in R_1$ and $0\in R_2$, there exist largest positive integers $s_1, 1\leq s_1 \leq k_1$ and $s_2, 1\leq s_2 \leq k_2$ such that for $1\leq \ell \leq s_1, x_{\ell}^2=0$ for some $x_{\ell}\in U_{i_{\ell}}$ and $1\leq r \leq s_2, y_r^2=0$ for some $y_r\in V_{j_r}$.

\begin{thm}\label{2}
Let $R_1$ and $R_2$ be any two rings with above properties, then $\chi(R_1)+\chi(R_2)-1\leq \chi(R_1\times R_2)\leq (\chi(R_1)-s_1)+(\chi(R_2)-s_2)+s_1s_2$. 
\end{thm}
\begin{proof}
Without loss generality, in $R_1$, let us assume that there is some $x_j\in U_j$ with $x_j^2$=0, for $1\leq j\leq s_1$ and there is no $x_j\in U_j$ with $x_j^2=0$, for $s_1+1\leq j\leq k_1$. Also we may assume that, in $R_2$, there is some $y_j\in V_j$ with $y_j^2=0$, for $1\leq j\leq s_2$ and there is no $y_j\in V_j$ with $y_j^2=0$, for $s_2+1\leq j\leq k_2$.  We now start the coloring. 

\noindent For $1\leq i\leq s_1$ and $1\leq j \leq s_2$, define  

 $c(x,y)=s_2(i-1)+j$, if $x\in U_i$ and $y\in V_j$. 

\noindent Next, for $1\leq i\leq s_1$ and $1\leq j \leq k_2-s_2$, define 

 $c(x,y)=s_1s_2+j$, if $x\in U_i$ and $y\in V_{s_2+j}$. 

\noindent Finally, for $1\leq i\leq k_1-s_1$, define 

 $c(x,y)=s_1s_2+(k_2-s_2)+i$, if $x\in U_{s_1+i}$ and $y\in R_2	$. 

\noindent Clearly, the function $c$ is a coloring of $R_1\times R_2$ onto $\{1,2,\ldots, s_1s_2+(k_1-s_1)+(k_2-s_2)\}$. 

\noindent To show $c$ is proper. Let $(x,y), (a,b)\in R_1\times R_2$ such that $(x,y)$ is adjacent to $(a,b)$, then $xa=0$ and $by=0$ and hence $x\in U_i$, $a\in U_j$, for some $1\leq i, j\leq k_1$ and $b\in V_{i^|}$ and $y\in V_{j^|}$ for some $1\leq i^|, j^| \leq k_2$. \\
\textbf{Case 1.} $x\neq a$.\\
Then $U_i\neq U_j$, otherwise $U_i$ would not be an independent subset of $R_1$. Note that the coloring of the vertices in $U_i\times \big(\bigcup\limits_{\ell=1}^{s_2}V_{\ell}\big)$ are different from the coloring of the vertices in $U_j\times \big(\bigcup\limits_{\ell=1}^{s_2}V_{\ell}\big)$. Also note that, for $s_2\leq \ell\leq k_2$, no two vertices in $\big(U_i\times V_{\ell}\big) \cup \big(U_j\times V_{\ell}\big)$ are adjacent. Further note that, for $s_2\leq \ell_1, \ell_2 \leq k_2$ with $\ell_1\neq \ell_2$, the coloring of the vertices in $U_i\times V_{\ell_1}$ are different from the coloring of the vertices in $U_j\times V_{\ell_2}$. Hence $c(x,y)\neq c(a,b)$.\\
\textbf{Case 2.} $x=a$.\\
Then $x^2=0$ and $U_i=U_j$. 
Clearly $y\neq b$ (otherwise $(x,y)=(a,b)$) and $y\in V_{i^{|}}$ and $b\in V_{j^{|}}$ where $i^{|}\neq j^{|}$ and then $c(x,y)\neq c(x,b)$. \\
In the similar way we can prove $c(x,y)\neq c(x,b)$ when $b=y$ or $b\neq y$. 
Thus $c$ is a proper coloring of $R_1\times R_2$. 
Hence $\chi(R_1\times R_2)\leq s_1s_2+(k_1-s_1)+(k_2-s_2)$.
\end{proof}


\noindent By induction on $n\geq 3$, we have, 

\begin{cor}\label{3}
For $1\leq i\leq n$, let $U^i_1,U^i_2,\ldots, U^i_{k_i}$ be a proper coloring of $R_i$ with $k_i=\chi(R_i)$. If for $1\leq i\leq n$, there is the largest positive integer $s_i, 1\leq s_i\leq k_i$ such that for $1\leq j\leq s_i, x_j^2=0$ for some $x_j\in U^i_{m_j}$. Then $\sum\limits_{i=1}^n\chi(R_i)-(n-1)\leq \chi(R_1\times R_2\times\ldots \times R_n)\leq \sum\limits_{i=1}^n(\chi(R_i)-s_i)+s_1s_2\ldots s_n$.

\end{cor}

\noindent Suppose no $x\in R_i$ with $x\neq 0$ and $x^2=0$ (that is, $R_i$ is a non-zero reduced ring), for all $1\leq i\leq n$, then $s_i=1$ and we have,

\begin{cor}\label{4}
If no $x\in R_i$ with $x\neq 0$ and $x^2=0$, (that is, $R_i$ is non-zero reduced ring), for $1\leq i\leq n$, then 
$\chi(R_1\times R_2\times\ldots \times R_n)=\sum\limits_{i=1}^n\chi(R_i)-(n-1)$.
\end{cor}

\begin{cor}
If $\textbf{R}$ is finite ring, then $R=R_1\times R_2\times\ldots\times R_n$, where $R_i$ is a finite local ring and hence $\sum\limits_{i=1}^n\chi(R_i)-(n-1)\leq \chi(R_1\times R_2\times\ldots \times R_n)\leq \sum\limits_{i=1}^n(\chi(R_i)-s_i)+s_1s_2\ldots s_n$.
\end{cor}


\section{Consequences}

\subsection{Some results of Beck's in \cite{Beck} are  consequences of our results.}

The following results were proved by Beck (in \cite{Beck}, Proposition 2.3, Theorem 5.5, Theorem 5.6 and Theorem 5.7). These results are consequence of our results. 

\begin{cor}[\cite{Beck}, Proposition 2.3]\label{5}
Let $p_1,\ldots, p_k, q_1,\ldots, q_r$ be different prime numbers
and $N=p_1^{2n_1}\ldots p_k^{2n_k}q_1^{2m_1+1}\ldots q_r^{2m_r+1}$. Then $\chi(\mathcal{Z}_N) =\omega(\mathcal{Z}_N)=p_1^{n_1}\ldots p_k^{n_k}q_1^{m_1}\ldots q_r^{m_r}+r$.
\end{cor}
\begin{proof}
We note that for $1\leq i\leq k$, $\mathcal{A}_i=\mathcal{B}_i \cup \mathcal{C}_i$ is a maximum clique of $\mathcal{Z}_{p_i^{2n_i}}$ and for $k+1\leq j\leq k+r$, $\mathcal{A}_j=\mathcal{B}_j \cup \mathcal{C}_j$ is a maximum clique of $\mathcal{Z}_{q_j^{2m_j+1}}$ (See in the  Section 2). It is easy to see that for $1\leq i\leq k$, $s_i=|\mathcal{B}_i|=p_i^{n_i}=\chi(\mathcal{Z}_{p_i^{2n_i}})=\omega(\mathcal{Z}_{p_i^{2n_i}})$, $|\mathcal{C}_i|=0$, 
 and for $k+1\leq j\leq k+r$, $s_j=|\mathcal{B}_j|=q_j^{m_j}=\chi(\mathcal{Z}_{q_j^{2m_j+1}})-1=\omega(\mathcal{Z}_{q_j^{2m_j+1}})-1$, $|\mathcal{C}_j|=1$. By Theorem \ref{1} and Corollary \ref{3}, we have $\omega(\mathcal{Z}_N)=p_1^{n_1}\ldots p_k^{n_k}q_1^{m_1}\ldots q_r^{m_r}+r=\chi(\mathcal{Z}_N)$. 
\end{proof}
\noindent Using Theorem \ref{2} and Corollary \ref{3}, we have the following consequence which were proved by Beck in \cite{Beck}.
\begin{cor}[\cite{Beck}, Theorem 5.5]\label{5.1}
A finite product of coloring is a coloring.
\end{cor}
\noindent Using Corollary \ref{5.1}, Beck proved the following results. 
\begin{cor}[\cite{Beck}, Theorem 5.6]
Let $I$ be a finitely generated ideal in a coloing. Then $R/Ann \ I$ is a coloring.
\end{cor}

\begin{cor}[\cite{Beck}, Corollary 5.7]
Let $R$ be a Noetherian ring whose nilradical is finite. Then $rad(Ann I)/Ann I$ is finite for any ideal.
\end{cor}

\subsection{Some main results of Anderson and Naseer in \cite{And} are  consequences of our results.}
Anderson and Naseer (in \cite{And}, Theorem 3.2 and its corollary 3.3) showed the following  results.
\begin{cor}[\cite{And}, Theorem 3.2]\label{7}
Let $R_1,\ldots R_k,R_{k+1},\ldots, R_{k+r}$ be colorings. Put $R=R_1\times \ldots R_k \times R_{k+1}\times \ldots \times R_{k+r}$. Let $J_i$ be the nilradical of $R_i$. 

\noindent (i) Suppose that the index of nilpotency for $J_i$ is $2n_i$ for $1\leq i\leq k$ and $2m_i-1$ for $k+1\leq i\leq k+r$ ($n_i, \ m_i \geq 1)$. \\
Then $\chi(R)\geq\omega(R)\geq |J_1^{n_1}|\ldots |J_k^{n_k}||J_{k+1}^{m_{k+1}}|\ldots |J_{k+r}^{m_{k+r}}|+r$.

\noindent (ii) Furthermore, suppose for each $i$, the following condition holds: For $1\leq i\leq k, x,y\in R_i$ with $xy=0$ implies $x\in J_i^{n_i}$ or $y\in J_i^{n_i}$ and if $x\notin J_i^{n_i+1}$, then $y\in J_i^{n_i}$. For $k+1\leq i\leq k+r, x,y\in R_i$ with $xy=0$ implies $x\in J_i^{m_i}$ or $y\in J_i^{m_i}$.\\
Then $\chi(R)=\omega(R)= |J_1^{n_1}|\ldots |J_k^{n_k}||J_{k+1}^{m_{k+1}}|\ldots |J_{k+r}^{m_{k+r}}|+r$
\end{cor}
Using the Corollary \ref{7}, they proved the following results.

\begin{cor}[\cite{And}, Corollary 3.3]\label{8}
Let $A$ be a regular Noetherian ring. Let $N_1,\ldots,N_k,N_{k+1},\ldots, N_{k+r}$ be maximul ideals of $A$ with each residue field $A/N_i$ finite. Consider the ring\\  $R=A/N_1^{2n_1}\ldots N_k^{2n_k} N_{k+1}^{2m_{k+1}-1}\ldots N_{k+r}^{2m_{k+r}-1}$. Then $R$ is a chromatic ring with $\chi(R)=\omega(R)=|A /N_1|^{\alpha_1} \ldots |A/N_k|^{\alpha_k} |A/N_{k+1}|^{\beta_{k+1}} \ldots |A/N_{k+r}|^{\beta_{k+r}}+r$, where $\alpha_i=\sum\limits_{\ell=n_i}^{2n_{i}-1}\binom{\ell-1+ht\ N_i}{ht\ N_i-1}$ and $\beta_i=\sum\limits_{\ell=m_i+1}^{2m_i-2}\binom{\ell-1+ht\ N_i}{ht\ N_i-1}$.
\end{cor}

\begin{cor}[\cite{And}, Corollary 3.5]\label{9}
Let $R$ be a coloring that is a finite direct product of reduced rings, principal ideal rings, and finite rings with index of nilpotency 2. Then $\chi(R)=\omega(R)$.
\end{cor}

The above results are consequence of the Theorem \ref{1} and Corollary \ref{3}, because of the following facts.
In Corollary \ref{7}, for $1\leq i \leq k,\ |\mathcal{B}_i|=|J_i^{n_i}|=s_i,\ |\mathcal{C}_i|=0$ and, for $k+1\leq j\leq k+r,\ |\mathcal{B}_j|=|J_j^{m_j}|=s_j,\ |\mathcal{C}_j|=1$. 

\subsection{Infinitely many counter examples of rings for the conjecture of Beck.}
In \cite{And}, they showed that the local ring $R=\mathcal{Z}_4[X,Y,Z]/M$, where $M$ is an ideal generated by $\big\{X^2-2, Y^2-2, Z^2-2, Z^2, 2X, 2Y,2Z, XY, XZ, YZ-2\big\}$ is a  counter example of the Conjecture 1.2. of Beck in \cite{Beck}. That is, $\omega(R)=5$ and $\chi(R)=6$. 
Here we construct infinitely many counter examples of the Beck's conjecture. We now recall a result of Beck in \cite{Beck}.

\begin{thm}[\cite{Beck}, Theorem 3.8]\label{6}
Let $R$ be a non-zero reduced ring which is coloring, then $R$ has only a finite  number of minimal prime ideals. If $r$ is this number, then $\chi(R)=\omega(R)=r+1$.
\end{thm}
Consider the ring $R_1=\mathcal{Z}_4[X,Y,Z]/ M$, where $M=\big<X^2-2, Y^2-2, Z^2-2, Z^2, 2X, $ \\ 
$2Y,2Z, XY, XZ, YZ-2\big>$.  
For $2\leq i\leq n$, no $x\in R_i$ with $x\neq 0$ and $x^2=0$ (that is, $R_i$ is a non-zero reduced ring). By Theorem \ref{1}, $\omega(R_1\times \ldots \times R_n)=5+\sum\limits_{i=2}^n\omega(R_i)-(n-1)$. By Corollary \ref{3}, 
$\chi(R_1\times \ldots \times R_n)=6+\sum\limits_{i=2}^n\chi(R_i)-(n-1)$. By Theorem \ref{6}, $\omega(R_i)=\chi(R_i)$, for $2\leq i\leq n$, we have $\omega(R_1\times \ldots \times R_n)<\chi(R_1\times \ldots \times R_n)$. 
\section{Concluding remark.}
In section 2, we found the clique number of finite product of rings in terms of its factor rings and in section 3, we obtained the bounds for chromatic number of finite product of rings in terms of its factor rings. 
Using these results, we obtained many consequences from \cite{And} and \cite{Beck} in sections 4.1 and 4.2. In section 4.3, we obtained infinitely many counter examples of the Conjecture 1.2. In this counter examples, the difference between chromatic number and clique is one, that is, $\chi-\omega$=1. We could not find examples of rings for which the difference is arbitrarily large. So we post the following question.

\noindent \textbf{Problem.} Given a positive integer $k\geq 2$, does there exist a ring $R$ for which $\chi(R)-\omega (R)=k$?

We believe that, the chromatic number of any ring is bounded by polynomial function of the clique number.

\subsection*{Acknowledgment}  This research was supported by the University Grant Commissions Start-Up Grant, Government of India grant No. F. 30-464/2019 (BSR) dated 27.03.2019.


\end{titlepage}
\end{document}